\documentclass[english, 10pt]{amsart}

\usepackage{babel}
\usepackage{amstext}
\usepackage{amsmath}
\usepackage{amssymb}
\usepackage{amsfonts}
\usepackage{latexsym}
\usepackage{ifthen}
\usepackage[all,2cell,dvips]{xy} \UseAllTwocells \SilentMatrices
\xyoption{all}
\usepackage{verbatim}

\newcommand{\rk}{{\rm rk}}

\newtheorem{lemma1}{}[section]

\newenvironment{lemma}{\begin{lemma1}{\bf Lemma.}}{\end{lemma1}}

\newenvironment{theorem}{\begin{lemma1}{\bf Theorem.}}{\end{lemma1}}

\newenvironment{corollary}{\begin{lemma1}{\bf Corollary.}}{\end{lemma1}}
\newenvironment{remark}{\begin{lemma1}{\bf Remark.}\rm}{\end{lemma1}}

\newenvironment{conjecture}{\begin {lemma1}{\bf Conjecture.}}{\end{lemma1}}

\newenvironment{remark*}{{\bf Remark.}}{}
\newenvironment{example*}{{\bf Example.}}{}

\newcommand{\Q}{\ensuremath{\mathbb{Q}}}

\newcommand{\C}{\ensuremath{\mathbb{C}}}
\newcommand{\N}{\ensuremath{\mathbb{N}}}
\newcommand{\PP}{\ensuremath{\mathbb{P}}}

\newcommand{\merom}[3]{\ensuremath{#1:#2 \dashrightarrow #3}}

\newcommand{\holom}[3]{\ensuremath{#1:#2  \rightarrow #3}}
\newcommand{\fibre}[2]{\ensuremath{#1^{-1} (#2)}}

\makeatletter
\ifnum\@ptsize=0  \fi \ifnum\@ptsize=2  \fi 

\newcommand\sF{{\mathcal F}}

\newcommand\sO{{\mathcal O}}

\DeclareMathOperator*{\red}{red}

\DeclareMathOperator*{\nons}{nons}

\DeclareMathOperator*{\Alb}{Alb}

\title{Manifolds with nef cotangent bundle} 
\date{November 15, 2011}

\author{Andreas H\"oring}

\thanks{Partially supported by the A.N.R. project ``CLASS''}

\address{Andreas H\"oring, Universit\'e Pierre et Marie Curie and Albert-Ludwig Universit\"at Freiburg}
\email{hoering@math.jussieu.fr}

\begin{document}

\begin{abstract}
Generalising a classical theorem by Ueno,
we prove structure results for manifolds with nef or semiample cotangent bundle.
\end{abstract}

\maketitle

\section{Introduction}

If $X$ is a submanifold of a complex torus, then by a classical
result of Ueno \cite[Thm.10.9]{Uen75} the manifold $X$ is an analytic fibre bundle with fibre a torus $T$
over a manifold $Y$ with ample canonical bundle. Moreover if $X$ is projective, then it decomposes (after 
finite \'etale cover) as a product $Y \times T$. Since for a submanifold of a complex torus the
cotangent bundle $\Omega_X$ is globally generated, it is natural to ask if there are analogues
of Ueno's result under a weaker positivity assumption.
Generalising a conjecture by Yau on compact K\"ahler manifolds with nonpositive bisectional curvature, 
Wu und Zheng  \cite{WZ02} proposed the following problem.

\begin{conjecture} \label{conjecturenef}
Let $X$ be a compact K\"ahler manifold with nef cotangent bundle $\Omega_X$. Then there exists a finite \'etale cover
$X' \rightarrow X$ such that the Iitaka fibration $X' \rightarrow Y'$ is a smooth fibration
onto a projective manifold $Y$ with ample canonical bundle and all the fibres are complex tori.
\end{conjecture}

In this note we prove this conjecture for projective manifolds with semiample canonical bundle, i.e.
some positive multiple $mK_X$ is generated by its global sections. 

\begin{theorem} \label{theoremmain}
Let $X$ be a projective manifold with nef cotangent bundle $\Omega_X$
and semiample canonical bundle $K_X$.
Then Conjecture \ref{conjecturenef} holds for $X$.
\end{theorem}

Since $K_X = \det \Omega_X$ is nef, the abundance conjecture \cite[Sec.2]{Rei85}  
claims that the semiampleness condition is redundant. So far this conjecture is known to hold if
$\dim X\leq 3$; see \cite{Uta92}. Note however that 
a projective manifold with nef cotangent bundle does not contain any rational curves, so the
abundance conjecture reduces to the weaker nonvanishing conjecture \cite[Thm.1.5]{a16}.
In particular our statement holds for fourfolds with $\kappa(X) \geq 0$. 

For manifolds with nonpositive bisectional curvature one expects the torus fibration
to be locally trivial \cite[p.264]{WZ02}. This is no longer true if we assume only that $\Omega_X$ is nef:
universal families over compact curves in the moduli space of abelian varieties (polarised and with level three structure) provide immediate counter-examples.
However if we assume that the cotangent bundle $\Omega_X$ itself is semiample we obtain
a precise analogue of Ueno's theorem:

\begin{theorem} \label{theoremsemiample}
Let $X$ be a projective manifold with semiample cotangent bundle, i.e. for some
positive integer $m \in \N$, the symmetric product $S^m \Omega_X$ is globally generated. 
Then there exists a finite \'etale cover
$X' \rightarrow X$ such that $X' \simeq Y \times A$ where $Y$ has ample canonical bundle and $A$ is an abelian variety.
\end{theorem}

This generalises a theorem of Fujiwara \cite[Thm.II]{Fuj92}.

While many of our arguments also work for compact K\"ahler manifolds, a crucial tool
is a theorem of Kawamata \cite[Thm.2]{Kaw91} which allows us to exclude the existence of higher-dimensional fibres.
In low dimension an elementary argument works also in the K\"ahler case (cf. Lemma \ref{lemmakawamatakaehler}), so we obtain:

\begin{theorem} \label{theoremkaehler}
Let $X$ be a compact K\"ahler manifold with nef cotangent bundle. If $\dim X \leq 3$, then 
Conjecture \ref{conjecturenef} holds for $X$.
\end{theorem}

This improves a result of Kratz \cite[Thm.1]{Kra97}. 

On a technical level the key point is that in our situation the tangent bundle is numerically flat
with respect to the Iitaka fibration. This allows to combine techniques used by Demailly, Peternell and Schneider 
in the study of manifolds with nef tangent bundles \cite{DPS94} with those introduced by Koll\'ar \cite{Kol93}
and Nakayama \cite{Nak99b} to understand torus fibrations.

{\bf Acknowledgements.} I want to thank Simone Diverio, Noburo Nakayama, Thomas Peternell and Maxime Wolff for helpful discussions.

\begin{center}
{\bf Notation}
\end{center}

We work over the complex field $\C$. For  positivity notions of vector bundles on compact K\"ahler
and projective varieties we refer to \cite{DPS94} and \cite{Laz04b}. 

A fibration is a proper surjective morphism \holom{\varphi}{X}{Y} with connected fibres
from a complex manifold onto a normal complex variety $Y$. 
We say that the fibration $\varphi$
\begin{itemize}
\item is almost smooth if for every $y \in Y$ the reduction $F_{\red}$ of the
fibre $F:=\fibre{\varphi}{y}$ is smooth and has the expected dimension;
\item is smooth in codimension one if there exists an analytic subset of codimension at
least two such that $(X \setminus \fibre{\varphi}{Z}) \rightarrow (Y \setminus Z)$ is a smooth fibration; 
\item has generically constant moduli if there exists a manifold $F_0$ such that every generic
fibre $F$ is isomorphic to $F_0$. By a theorem of Fischer and Grauert \cite{FG65} this is equivalent 
to the property that $\varphi$ is locally trivial over some Zariski open set.
\end{itemize}

If \holom{\varphi}{X}{Y} is a fibration and $\mu: X' \rightarrow X$ a finite \'etale cover, 
there exists a fibration $\holom{\varphi'}{X'}{Y'}$ and a finite map
$\holom{\mu'}{Y'}{Y}$ such that $\varphi \circ \mu=\mu' \circ \varphi'$.
Since we never consider $\holom{\mu'}{Y'}{Y}$ we call the fibration $\holom{\varphi'}{X'}{Y'}$
the Stein factorisation (of $\varphi$ and $\mu$).

\section{A structure result for fibrations}

Recall that a vector bundle $E$ on a compact K\"ahler variety is numerically flat \cite[Defn.1.17]{DPS94}
if both $E$ and $E^*$ are nef. This is equivalent to the property that $E$ is nef and $\det E$ is numerically trivial,
i.e. $c_1(E)=0$. 

If \holom{\varphi}{X}{Y} is a fibration from a K\"ahler manifold onto a normal variety and $E$ a vector bundle on $X$,
we say that $E$ is $\varphi$-nef (resp. $\varphi$-numerically flat) if this property holds for any variety $Z \subset Y$
that is contracted by $\varphi$, i.e. such that $\varphi(Z)=pt$.
We note that if the cotangent bundle $\Omega_X$ is $\varphi$-nef, 
then any subvariety $Z \subset X$ contracted by $\varphi$ has nef cotangent sheaf: 
indeed $\Omega_Z$ is a quotient of $\Omega_X|_Z$, so it is nef.
Moreover in this case $\varphi$ does not contract any rational curves: if $f: \PP^1 \rightarrow X$
is a non-constant morphism such that $\varphi \circ f$ is constant, 
the tangent map gives a non-zero map $f^* \Omega_X \rightarrow \Omega_{\PP^1} \simeq \sO_{\PP^1}(-2)$, which violates the nefness assumption.

\begin{lemma} \label{lemmaalmostsmooth}
Let $X$ be a K\"ahler manifold that admits an equidimensional fibration \holom{\varphi}{X}{Y}
onto a normal variety $Y$ such that the tangent bundle $T_X$ is $\varphi$-numerically flat.
Then the following holds:
\begin{enumerate}
\item The fibration $\varphi$ is almost smooth. Moreover every set-theoretical fibre $F_{ \red}$ 
is a finite \'etale quotient $T \rightarrow F_{\red}$ of a torus $T$.
\item There exists a finite \'etale cover $X' \rightarrow X$ such that the Stein factorisation
\holom{\varphi'}{X'}{Y'} is smooth in codimension one and the smooth fibres are tori. 
\item If moreover $X$ is projective, there exists a finite \'etale cover $X' \rightarrow X$ such that the Stein factorisation
\holom{\varphi'}{X'}{Y'} is an abelian group scheme. 
\item If $X$ is compact and $\varphi$ has generically constant moduli, there exists a finite \'etale cover $X' \rightarrow X$
such that the Stein factorisation \holom{\varphi'}{X'}{Y'}  is smooth and locally trivial.
If moreover $X$ is projective, then (after finite \'etale cover) one has $X' \simeq Y' \times A$
with $A$ an abelian variety.
\end{enumerate}
\end{lemma}

\begin{remark*}
The statement does not generalise to non-K\"ahler manifolds. In fact there are examples of compact non-K\"ahler surfaces $X$
admitting an elliptic fibration onto $\PP^1$ that is almost smooth with a unique singular fibre. 
Arguing as in \cite[V.13.2]{BHPV04} one sees that one cannot remove the multiple fibre by an \'etale cover $X' \rightarrow X$.
\end{remark*}

\begin{proof} 

{\bf Step 1: $\varphi$ almost smooth in codimension one, i.e. there exists a subvariety $Z \subset Y$
of codimension at least two such that $(X \setminus \fibre{\varphi}{Z}) \rightarrow (Y \setminus Z)$ is almost smooth.}

We argue by contradiction. 
Choosing a generic disc that meets a codimension one component of the $\varphi$-singular locus in a generic point, we reduce the problem to the case where $Y$ is a curve. Let $F$ be a fibre such that the reduction
$F_{\red}$ is not smooth. We decompose the divisor $F=\sum_{i=1}^k a_i F_i$ where the $F_i$ are pairwise distinct prime divisors. Since $F_1$ is contained in a $\varphi$-fibre, the bundle $\Omega_X|_{F_1}$ is numerically flat.
Thus its quotient $\Omega_{F_1}$ is nef, so on the one hand the dualising sheaf $\omega_{F_1}$ is nef. 
On the other hand by adjunction one has $\omega_{F_1} \simeq (\omega_X \otimes \sO_{X}(F_1))|_{F_1}$.
Since $\omega_X|_{F_1}$ and $\sO_{X}(F)|_{F_1}$ are numerically trivial, we see that
$$
\omega_{F_1} \sim_\Q \sO_{F_1}(-\sum_{i=2}^k \frac{a_i}{a_1} (F_i \cap F_1)).
$$
Thus $\omega_{F_1}$ is nef and anti-effective, hence trivial. By connectedness of the fibre, we have $k=1$, i.e.
$F$ is irreducible. Since $T_X$ is $\varphi$-nef, a result of Demailly-Peternell-Schneider 
\cite[Prop.5.1]{DPS94} (see also Remark \ref{remarkgap}) now shows that $F_{\red}$ is smooth, a contradiction. 

Thus $\varphi$ is almost smooth in codimension one, and if $F$ is a fibre such that $F_{\red}$ is smooth,
its normal bundle $N_{F_{\red}/X}$ is numerically flat \cite[Prop.5.1]{DPS94}.
In particular by adjunction $K_{F_{\red}} \equiv 0$ and as we have seen above, the cotangent bundle $\Omega_{F_{\red}}$ is nef.  
The Chern class inequalities \cite[Thm.2.5.]{DPS94}
$$
0 = c_1^2(\Omega_{F_{\red}}) \geq c_2(\Omega_{F_{\red}}) \geq 0
$$
show that $c_2(F_{\red})=0$. Thus a classical result of Bieberbach \cite[Cor.4.15]{Kob87} shows that $F_{\red}$ is a finite \'etale quotient of a torus.

{\bf Step 2: Proof of Statement 2).}
Let $N \subset Y$ be a subvariety of codimension at least two. Since $\varphi$ is equidimensional, $\fibre{\varphi}{N}$ has codimension
at least two. Hence we have an isomorphism of fundamental groups
$\pi_1(X) \simeq \pi_1(X \setminus \fibre{\varphi}{N})$ and any 
\'etale cover $(X \setminus \fibre{\varphi}{N})'  \rightarrow (X \setminus \fibre{\varphi}{N})$ extends to an \'etale cover
$X' \rightarrow X$. Thus by Step 1) we can suppose without loss of generality that we are in the situation of the following 
lemma. 

\begin{lemma} \label{lemmalifting}
Let $\holom{\varphi}{X}{Y}$ be an almost smooth fibration from a K\"ahler manifold $X$ onto a manifold $Y$.
Suppose that $\varphi$ is smooth in the complement of a smooth divisor $D \subset Y$.
Suppose moreover that for every fibre $F$, the set-theoretical fibre
$F_{\red}$ is a finite \'etale quotient $T \rightarrow F_{\red}$ of a torus $T$.
Then there exists a finite \'etale cover $X' \rightarrow X$ such that the Stein factorisation
\holom{\varphi'}{X'}{Y'} is smooth in codimension one and the smooth fibres are tori. 
\end{lemma}

\begin{remark*}
This result is certainly well-known to experts. In fact the fibration being almost smooth, 
the local monodromies of the variation of Hodge structures 
around $D$ are finite. The existence of the cover $X' \rightarrow X$ 
then follows analogously to the proof of \cite[Thm.6.3]{Kol93}. 
For the convenience of the reader we follow an argument indicated by Noburo Nakayama.
\end{remark*}

\begin{proof}[Proof of Lemma \ref{lemmalifting}]
We can cover $Y$ by polydiscs $\Delta$ of dimension $m:=\dim Y$ such that 
$$
\Delta \cap D = \{ (w_1, \ldots, w_m) \in \Delta \ | \ w_m = 0 \}
$$
and for $y  \in \Delta \cap D$ and $x \in \fibre{\varphi}{\Delta}$ there exist local coordinates $z_1, \ldots, z_n$ around $x$
such that $\varphi$ is given by $(z_1, \ldots, z_n) \rightarrow (z_1, \ldots, z_{m-1}, z_m^k)$, where $k$
is the multiplicity of the fibre $F$. Let $\Delta' \rightarrow \Delta$ be a finite map from some $m$-dimensional disc $\Delta'$
that ramifies exactly along $\Delta \cap D$ with multiplicity $k$.
Let $X_{\Delta'}$ be the normalisation 
of the fibre product $\Delta' \times_\Delta X$, then a local computation shows that
$X_{\Delta'} \rightarrow \fibre{\varphi}{\Delta} \subset X$ is \'etale
and the fibration $X_{\Delta'} \rightarrow \Delta'$ is smooth. Since $\Delta'$ retracts onto a point
we have an isomorphism $\pi_1(F) \simeq \pi_1(X_{\Delta'})$, where $F$ is any fibre.
The cover $X_{\Delta'} \rightarrow \fibre{\varphi}{\Delta}$ being \'etale and surjective this shows that we have an injection
$$
\pi_1(F) \hookrightarrow \pi_1(\fibre{\varphi}{\Delta}).
$$
By \cite[Thm.7.8]{Nak99b} this implies that $\varphi$ 
is bimeromorphically equivalent to a fibration $\holom{\tilde \varphi}{\tilde X}{\tilde Y}$
which becomes smooth after a finite \'etale cover. As we have just seen for such a
fibration the natural morphism $\pi_1(\tilde F) \rightarrow \pi_1(\tilde X)$ is injective.
Since the fibrations $\varphi$ and $\tilde \varphi$ are bimeromorphic, this shows that
$$
\pi_1(F) \rightarrow \pi_1(X)
$$ 
is injective. Thus by \cite[Thm.8.6]{Nak99b} (which is the analogue of \cite[Thm.6.3]{Kol93} for the K\"ahler case) there
exists a finite \'etale cover $X' \rightarrow X$ such that the Stein factorisation 
$\holom{\varphi'}{X'}{Y'}$ is bimeromorphically equivalent to a smooth torus fibration 
$\holom{\tilde \varphi}{\tilde X}{\tilde Y}$. 
Up to blowing up $\tilde Y$ and excluding the image of the exceptional locus we can suppose without loss of generality that $\tilde Y= Y'$.  Since in codimension one the $\varphi'$-fibres
do not contain any rational curves, there exists a codimension two set $B \subset Y'$
such that the restriction of the 
bimeromorphic map \merom{\mu}{\tilde X}{X'} to $\tilde X \setminus \fibre{\tilde \varphi}{B}$ is a morphism and an isomorphism onto its image. Since $\tilde \varphi$ is smooth, this proves the statement.
\end{proof}

{\bf Step 3: $\varphi$ is almost smooth.} This property does not change under finite \'etale cover, so we can assume
by Step 2) that $\varphi$ is smooth in codimension one. Moreover $\varphi$ is equidimensional, so
the relative cotangent sheaf $\Omega_{X/Y}$ is locally free in codimension one and has determinant
$\sO_X(K_{X/Y})$. We consider the foliation $\sF \subset T_X$ defined by the reduction of every
$\varphi$-fibre $F$, i.e. on the non-singular locus $F_{\red, \nons} \subset F_{\red}$
we have
$$
(*) \qquad T_{F_{\red, \nons}} = \sF|_{F_{\red, \nons}}.
$$
Since $\varphi$ is smooth in codimension one, the sheaves $T_{X/Y}:=\Omega_{X/Y}^*$ and $\sF$
coincide in codimension one, hence $\det \sF \simeq \sO_X(-K_{X/Y})$. We claim that the foliation $\sF$
is regular which obviously implies that the reduction of every fibre is smooth.

{\em Proof of the claim.}
The inclusion $\sF \subset T_X$ 
induces a map $\alpha: \det \sF \rightarrow \bigwedge^{\rk \sF} T_X$ and
by \cite[Lemma 1.20]{DPS94} it is sufficient to show that $\alpha$ has rank one in every point.
By $(*)$ the restriction of $\alpha$ 
to $F_{\red, \nons}$ identifies to the map induced by $T_{F_{\red, \nons}} \subset T_X|_{F_{\red, \nons}}$,
hence  $\alpha|_{F_{\red}}$ is not zero on any irreducible component of $F_{\red}$. 
Since $\det \sF \simeq  \sO_X(-K_{X/Y})$ and $\bigwedge^{\rk \sF} T_X$ are $\varphi$-numerically flat, 
we know by \cite[Prop.1.2(12)]{CP91} that 
$\alpha|_{F_{\red}}$ does not vanish in any point of $F_{\red}$. 
Thus $\alpha$ does not vanish in any point of $X$.

{\bf Step 4: Proof of Statement 3).}
By what precedes we know that $\varphi$ is almost smooth and (after finite \'etale cover) smooth
in codimension one. Since $X$ is projective we know by \cite[Thm.6.3]{Kol93}
that (after finite \'etale cover) the fibration $\varphi$ is birational to an abelian group scheme
$\tilde \varphi: \tilde X \rightarrow \tilde Y$. 
Since $\tilde \varphi$ is a group scheme, there exists a section $s: \tilde Y \rightarrow \tilde X$.
Let $Z$ be the strict transform of $s(\tilde Y)$ under the birational map
$\tilde X \dashrightarrow X$. Then $\varphi|_Z: Z \rightarrow Y$ is birational, 
i.e. $Z$ is generically a section of $\varphi$.
In particular for a general fibre $F$ we have $F \cdot Z=1$. Since for any fibre $F_0$ we
have $[F_0]=m [F]$ with $m$ the multiplicity of the fibre $F_0$, we see that all the fibres are reduced.
Thus the almost smooth fibration $\varphi$ is smooth. 

{\bf Step 5: Proof of Statement 4).}
By Statements 1) and 2) we know that  (after finite \'etale cover) the almost smooth fibration $\varphi$ 
has tori as general fibres. If $\varphi$ has generically constant moduli,
we have (after finite \'etale cover) that $q(X)=q(Y)+\dim F$ \cite[Prop.6.7]{CP00}.
Since the reduction of every $\varphi$-fibre is an \'etale quotient of a torus, the
Albanese map $\alpha_X: X \rightarrow \Alb(X)$ maps
each $\varphi$-fibre isomorphically onto a fibre of the locally trivial fibration $\varphi_*: \Alb(X) \rightarrow \Alb(Y)$.
By the universal property of the fibre product we have a commutative diagram
$$
\xymatrix{
\Alb(X) \times_{\Alb(Y)} Y   \ar[rd]_\psi 
& X  \ar[l] \ar[r]^{\alpha_X} \ar[d]_{\varphi} & \Alb(X) \ar[d]^{\varphi_*}
\\
& Y \ar[r]^{\alpha_Y}  & \Alb(Y)
}
$$
The map $\psi$ is the pull-back of $\varphi_*$ by the fibre product, so it is a locally trivial fibration. The base $Y$ is normal,
so the total space $\Alb(X) \times_{\Alb(Y)} Y$ is normal.
By what precedes the morphism $X \rightarrow \Alb(X) \times_{\Alb(Y)} Y$ is bimeromorphic and finite, hence
an isomorphism by Zariski's main theorem. In particular $\varphi=\psi$ is smooth and locally trivial.

If $X$ is projective, the Albanese torus is an abelian variety.
Thus we know by Poincar\'e's reducibility theorem \cite[Thm.5.3.5]{BL04} that (after finite \'etale cover) one has $\Alb(X) \simeq \Alb(Y) \times F$, hence the fibre product $\Alb(X) \times_{\Alb(Y)} Y$ is isomorphic to $Y \times F$.
\end{proof}


As a corollary of the proof we obtain the following statement.

\begin{corollary} 
Let $X$ be a K\"ahler manifold that admits an equidimensional almost smooth fibration \holom{\varphi}{X}{Y}
onto a normal variety $Y$ such that the general fibre $F$ is a finite \'etale quotient $T \rightarrow F$ 
of a torus $T$. Then there exists a finite \'etale cover $X' \rightarrow X$ such that the Stein factorisation
\holom{\varphi'}{X'}{Y'} is smooth in codimension one and the smooth fibres are tori. 
\end{corollary}

Note that by \cite[Lemme 2.3]{Cla10} for a torus fibration that is smooth in codimension one
the map $\pi_1(F) \rightarrow \pi_1(X)$ is injective. Thus in the situation above
$\varphi$ has generically large fundamental group along the general fibre \cite[Defn.6.1]{Kol93}, i.e.
the statement is a natural inverse to \cite[Thm.6.3]{Kol93}.

We can also deduce a simplified version of \cite[Prop.5.1]{DPS94}:

\begin{corollary} \label{corollarydps}
Let $X$ be a quasi-projective manifold that admits a fibration \holom{\varphi}{X}{Y}
onto a normal variety $Y$ such that the tangent bundle $T_X$ is $\varphi$-nef.
Then $\varphi$ is equidimensional and almost smooth.
If $X$ is projective, there exists a finite \'etale cover $X' \rightarrow X$ such that the Stein factorisation \holom{\varphi'}{X'}{Y'}  is smooth.
\end{corollary}

\begin{remark} \label{remarkgap}
Sol\'a Conde and Wi\'sniewski \cite[Ch.4.2]{SW04} point out that the proof of the ``First case'' of \cite[Prop.5.1]{DPS94} has a gap and give a completely different proof under the additional condition that $\varphi$ is a Mori contraction
\cite[Thm.4.4]{SW04}.
Note that we used \cite[Prop.5.1]{DPS94} in the proof of Lemma \ref{lemmaalmostsmooth}, but only
for a fibration over a curve which corresponds to the ``Second case'' of their proof.
\end{remark}

\begin{proof}
If $K_X$ is not $\varphi$-nef, we know by the relative contraction theorem \cite[Thm.3.25]{KM98} that there exists 
an elementary Mori contraction $\holom{\mu}{X}{Z}$ that factors $\varphi$, i.e.
there exists a fibration $\psi: Z \rightarrow Y$ such that $\varphi=\psi \circ \mu$. 
Applying \cite[Thm.4.4]{SW04} to $\mu$ we see that $\mu$ and $Z$ are smooth, in
particular $T_Z$ is $\psi$-nef. Since a composition of equidimensional and almost smooth fibrations is equidimensional and 
almost smooth, we can argue inductively and suppose without loss of generality that $K_X$ is $\varphi$-nef.
Since $T_X$ is also $\varphi$-nef, it is $\varphi$-numerically flat. Hence
$\Omega_X$ is also $\varphi$-nef, so the $\varphi$-fibres do not contain any rational curves.
By a theorem of Kawamata \cite[Thm.2]{Kaw91} this shows that $\varphi$ is equidimensional.
Conclude by Lemma \ref{lemmaalmostsmooth},1) and 3). 
\end{proof}

\section{Proofs of the main statements}

\begin{proof}[Proof of Theorem \ref{theoremmain}.]
Since $K_X$ is semiample we can consider the Iitaka fibration $\varphi: X \rightarrow Y$.
Note that the anticanonical divisor $-K_X$ is $\varphi$-numerically trivial.
Since $\Omega_X$ is nef, hence $\varphi$-nef, the tangent bundle $T_X$
is $\varphi$-numerically flat. By Corollary \ref{corollarydps} the fibration $\varphi$ is equidimensional. 
We conclude by Lemma \ref{lemmaalmostsmooth},3) that there exists a finite \'etale cover
such that the Iitaka fibration is an abelian group scheme. 
By \cite[5.9.1]{Kol93} the projective manifold $Y$ is of general type, so in order to see that
$K_Y$ is ample it is sufficient to show that $Y$ does not contain any rational curves\footnote{
This is a well-known consequence of cone theorem, base-point free theorem and \cite[Thm.2]{Kaw91}.}.
Yet the abelian group scheme $X \rightarrow Y$ has a section, so any rational curve $\PP^1 \rightarrow Y$
would lift to $X$. This is excluded by the nefness of $\Omega_X$.
\end{proof}

\begin{proof}[Proof of Theorem \ref{theoremsemiample}.]
By Theorem \ref{theoremmain} we can suppose (after finite \'etale cover) that the Iitaka fibration $\varphi: X \rightarrow
Y$ is smooth with abelian fibres. Let $F=\fibre{\varphi}{y}$ be any smooth fibre, then we have an exact sequence
$$
0 \rightarrow (\varphi^* \Omega_Y)|_F \rightarrow \Omega_X|_F \rightarrow \Omega_F \simeq \sO_F^{\oplus \dim F}
\rightarrow 0.
$$
Since $\Omega_X|_F$ is semiample and $\det \Omega_F$ is trivial we know by \cite[Cor.4]{Fuj92} that 
the exact sequence splits. In particular the Kodaira spencer map is zero in $y$. Since this holds for all $y$
we see that $\varphi$ has constant moduli. Conclude by Lemma \ref{lemmaalmostsmooth},4).
\end{proof}

Before we can prove Theorem \ref{theoremkaehler} we need a technical lemma which
is a first step towards a generalisation of \cite[Thm.2]{Kaw91} to the K\"ahler case.

\begin{lemma} \label{lemmakawamatakaehler}
Let $X$ be a compact K\"ahler threefold and let $\holom{\varphi}{X}{S}$ be 
a fibration onto a projective surface such that $-K_X$ is $\varphi$-nef. Let $D \subset X$ be a divisor that is contracted
by $\varphi$. Then $D$ is uniruled. 
\end{lemma}

\begin{remark*}
The proof is based on the fact that a compact K\"ahler surface $D$ with Gorenstein singularities
is uniruled if the canonical sheaf $\omega_D$ is not pseudoeffective. This is well-known if $D$ is smooth
and standard arguments (cf. the proof of \cite[Lemma 4.2]{a16}) allow to generalise to singular $D$.
Note that for projective manifolds the implication 
$$
K_D \ \mbox{not pseudoeffective} \ \Rightarrow \  D \ \mbox{uniruled} 
$$
is a famous theorem \cite{BDPP04}, but for K\"ahler manifolds this is only known up to dimension three \cite{Bru06}.
\end{remark*}

\begin{proof}
We fix a K\"ahler form $\alpha$ on $X$.
Let $H$ be an effective divisor passing through $\varphi(D)$, then
we can write $\varphi^* H=H'+mD$ with $m \in \N$ and $D \not\subset \mbox{supp} H'$ but $D \cap H' \neq 0$.
Since $\varphi^* H \cdot D = 0$ we have
\[
\alpha \cdot (\varphi^* H)^2 = \alpha \cdot \varphi^* H \cdot (H' + mD) = \alpha \cdot (H')^2 + \alpha \cdot H' \cdot mD,
\]
and developing the left hand side implies $\alpha \cdot H' \cdot mD = - \alpha \cdot m^2 D^2$.
Since $H' \cap D$ is an effective non-zero cycle, we see that $\alpha \cdot D^2<0$.
By the adjunction formula we have $\omega_D \simeq \sO_D(K_X+D)$, so our computation shows that
$$
\omega_D \cdot \alpha|_D = (K_X+D) \cdot D \cdot \alpha < 0.
$$
Therefore $\omega_D$ is not pseudoeffective, hence $D$ is uniruled.
\end{proof}

\begin{proof}[Proof of Theorem \ref{theoremkaehler}.]
Since $K_X$ is nef and $\dim X \leq 3$, it is semiample \cite[Thm.1]{Pet01}, \cite{DP03}.
Let $\varphi: X \rightarrow Y$ be the Iitaka fibration, then the anticanonical divisor $-K_X$ is $\varphi$-trivial.
Since $\Omega_X$ is nef, hence $\varphi$-nef and $-K_X$ is $\varphi$-trivial, the tangent bundle $T_X$
is $\varphi$-trivial. If $\dim Y=1$ we conclude by Lemma \ref{lemmaalmostsmooth},2).
The cases $\dim Y=0$ or $3$ being trivial, we are left with case $\dim Y=2$:

By Lemma \ref{lemmakawamatakaehler}, the fibration $\varphi$ is equidimensional.
Thus it is almost smooth and (after finite \'etale cover) smooth in codimension one by Lemma \ref{lemmaalmostsmooth},2). Since every complete family of elliptic curves is isotrivial,
$\varphi$ has generically constant moduli. We conclude by Lemma \ref{lemmaalmostsmooth},4).
\end{proof}


\begin{thebibliography}{BHPVdV04}

\bibitem[BDPP04]{BDPP04}
S{\'e}bastien Boucksom, Jean-Pierre Demailly, Mihai P{\u a}un, and Thomas
  Peternell.
\newblock The pseudo-effective cone of a compact {K\"a}hler manifold and
  varieties of negative {K}odaira dimension.
\newblock {\em arxiv preprint, to appear in J.A.G.}, 2004.

\bibitem[BHPVdV04]{BHPV04}
Wolf~P. Barth, Klaus Hulek, Chris A.~M. Peters, and Antonius Van~de Ven.
\newblock {\em Compact complex surfaces}, volume~4 of {\em Ergebnisse der
  Mathematik und ihrer Grenzgebiete. 3. Folge.}
\newblock Springer-Verlag, Berlin, second edition, 2004.

\bibitem[BL04]{BL04}
Christina Birkenhake and Herbert Lange.
\newblock {\em Complex abelian varieties}, volume 302 of {\em Grundlehren der
  Mathematischen Wissenschaften}.
\newblock Springer-Verlag, Berlin, second edition, 2004.

\bibitem[Bru06]{Bru06}
Marco Brunella.
\newblock A positivity property for foliations on compact {K}\"ahler manifolds.
\newblock {\em Internat. J. Math.}, 17(1):35--43, 2006.

\bibitem[Cla10]{Cla10}
Beno{\^{\i}}t Claudon.
\newblock Invariance de la {$\Gamma$}-dimension pour certaines familles
  k\"ahl\'eriennes de dimension 3.
\newblock {\em Math. Z.}, 266(2):265--284, 2010.

\bibitem[CP91]{CP91}
Fr{\'e}d{\'e}ric Campana and Thomas Peternell.
\newblock Projective manifolds whose tangent bundles are numerically effective.
\newblock {\em Math. Ann.}, 289(1):169--187, 1991.

\bibitem[CP00]{CP00}
Fr{\'e}d{\'e}ric Campana and Thomas Peternell.
\newblock Complex threefolds with non-trivial holomorphic {$2$}-forms.
\newblock {\em J. Algebraic Geom.}, 9(2):223--264, 2000.

\bibitem[DP03]{DP03}
Jean-Pierre Demailly and Thomas Peternell.
\newblock A {K}awamata-{V}iehweg vanishing theorem on compact {K}\"ahler
  manifolds.
\newblock {\em J. Differential Geom.}, 63(2):231--277, 2003.

\bibitem[DPS94]{DPS94}
Jean-Pierre Demailly, Thomas Peternell, and Michael Schneider.
\newblock Compact complex manifolds with numerically effective tangent bundles.
\newblock {\em J. Algebraic Geom.}, 3(2):295--345, 1994.

\bibitem[FG65]{FG65}
Wolfgang Fischer and Hans Grauert.
\newblock Lokal-triviale {F}amilien kompakter komplexer {M}annigfaltigkeiten.
\newblock {\em Nachr. Akad. Wiss. G\"ottingen Math.-Phys. Kl. II}, 1965:89--94,
  1965.

\bibitem[Fuj92]{Fuj92}
Tsuyoshi Fujiwara.
\newblock Varieties of small {K}odaira dimension whose cotangent bundles are
  semiample.
\newblock {\em Compositio Math.}, 84(1):43--52, 1992.

\bibitem[HPR11]{a16}
Andreas H{\"o}ring, Thomas Peternell, and Ivo Radloff.
\newblock Uniformisation in dimension four: towards a conjecture of {I}itaka.
\newblock {\em arXiv preprint}, 1103.5392v1, 2011.

\bibitem[Kaw91]{Kaw91}
Yujiro Kawamata.
\newblock On the length of an extremal rational curve.
\newblock {\em Invent. Math.}, 105(3):609--611, 1991.

\bibitem[KM98]{KM98}
J{\'a}nos Koll{\'a}r and Shigefumi Mori.
\newblock {\em Birational geometry of algebraic varieties}, volume 134 of {\em
  Cambridge Tracts in Mathematics}.
\newblock Cambridge University Press, Cambridge, 1998.
\newblock With the collaboration of C. H. Clemens and A. Corti.

\bibitem[Kob87]{Kob87}
Shoshichi Kobayashi.
\newblock {\em Differential geometry of complex vector bundles}, volume~15 of
  {\em Publications of the Mathematical Society of Japan}.
\newblock Princeton University Press, Princeton, NJ, 1987.
\newblock Kan{\^o} Memorial Lectures, 5.

\bibitem[Kol93]{Kol93}
J{\'a}nos Koll{\'a}r.
\newblock Shafarevich maps and plurigenera of algebraic varieties.
\newblock {\em Invent. Math.}, 113(1):177--215, 1993.

\bibitem[Kra97]{Kra97}
Henrik Kratz.
\newblock Compact complex manifolds with numerically effective cotangent
  bundles.
\newblock {\em Doc. Math.}, 2:183--193 (electronic), 1997.

\bibitem[Kwc92]{Uta92}
J{\'a}nos Koll{\'a}r~(with 14~coauthors).
\newblock {\em Flips and abundance for algebraic threefolds}.
\newblock Soci\'et\'e Math\'ematique de France, Paris, 1992.
\newblock Papers from the Second Summer Seminar on Algebraic Geometry held at
  the University of Utah, Salt Lake City, Utah, August 1991, Ast{\'e}risque No.
  211 (1992).

\bibitem[Laz04]{Laz04b}
Robert Lazarsfeld.
\newblock {\em Positivity in algebraic geometry. {II}}, volume~49 of {\em
  Ergebnisse der Mathematik und ihrer Grenzgebiete.}
\newblock Springer-Verlag, Berlin, 2004.
\newblock Positivity for vector bundles, and multiplier ideals.

\bibitem[Nak99]{Nak99b}
Noboru Nakayama.
\newblock Compact {K\"a}hler manifolds whose universal covering spaces are
  biholomorphic to {${\bf C}^n$}.
\newblock {\em RIMS preprint}, 1230, 1999.

\bibitem[Pet01]{Pet01}
Thomas Peternell.
\newblock Towards a {M}ori theory on compact {K}\"ahler threefolds. {III}.
\newblock {\em Bull. Soc. Math. France}, 129(3):339--356, 2001.

\bibitem[Rei87]{Rei85}
Miles Reid.
\newblock Tendencious survey of {$3$}-folds.
\newblock In {\em Algebraic geometry, {B}owdoin, 1985 ({B}runswick, {M}aine,
  1985)}, volume~46 of {\em Proc. Sympos. Pure Math.}, pages 333--344. Amer.
  Math. Soc., Providence, RI, 1987.

\bibitem[SCW04]{SW04}
Luis~Eduardo Sol{\'a}~Conde and Jaros{\l}aw~A. Wi{\'s}niewski.
\newblock On manifolds whose tangent bundle is big and 1-ample.
\newblock {\em Proc. London Math. Soc. (3)}, 89(2):273--290, 2004.

\bibitem[Uen75]{Uen75}
Kenji Ueno.
\newblock {\em Classification theory of algebraic varieties and compact complex
  spaces}.
\newblock Springer-Verlag, Berlin, 1975.
\newblock Notes written in collaboration with P. Cherenack, Lecture Notes in
  Mathematics, Vol. 439.

\bibitem[WZ02]{WZ02}
Hung-Hsi Wu and Fangyang Zheng.
\newblock Compact {K}\"ahler manifolds with nonpositive bisectional curvature.
\newblock {\em J. Differential Geom.}, 61(2):263--287, 2002.

\end{thebibliography}
\end{document}